\documentclass[11pt,reqno]{amsart}
\usepackage{amsmath}
\usepackage{amsthm}
\usepackage{amssymb}
\usepackage{graphicx}
\usepackage{amsfonts}
\usepackage{latexsym}
\usepackage{setspace}
\usepackage{hyperref}

\newcommand{\J}{ \mathbb{J}}

\newcommand{\norm}[1]{\left\| #1 \right\|}
\newcommand{\inner}[1]{\langle #1 \rangle}

\newcommand{\E}{\mathcal{E}}

\newcommand{\N}{\mathbb{N}}

\newcommand{\h}{\mathcal{H}}

\newcommand{\minimatrix}[4]{\begin{pmatrix} #1 & #2 \\ #3 & #4 \end{pmatrix}  }

\newcommand{\twovector}[2]{\begin{pmatrix} #1\\#2 \end{pmatrix} }

\renewcommand{\phi}{\varphi}

\newtheorem{Corollary}{Corollary}
\newtheorem{Theorem}{Theorem}

\newtheorem{Lemma}{Lemma}
\theoremstyle{definition}
\newtheorem*{Definition}{Definition}

\newtheorem{Example}{Example}

\begin{document}
    \title{The Norm and Modulus of a Foguel Operator}

    \author{Stephan Ramon Garcia}
    \address{   Department of Mathematics\\
            Pomona College\\
            Claremont, California\\
            91711 \\ USA}
    \email{Stephan.Garcia@pomona.edu}
    \urladdr{http://pages.pomona.edu/\textasciitilde sg064747}
    
    
    \thanks{Partially supported by National Science Foundation Grant DMS-0638789.}

    \begin{abstract}
    	We develop a method for calculating the norm and the spectrum of the modulus
	of a Foguel operator.  In many cases, the norm can be computed exactly.  In others,
	sharp upper bounds are obtained.  In particular, we observe several connections between
	Foguel operators and the Golden Ratio.
    \end{abstract}

\maketitle

\section{Introduction}

\begin{Definition}
	Let $\E$ denote a separable Hilbert space and let $S$
	denote the unilateral shift on $l^2_{\E}(\N)$.
	A \emph{Foguel operator} is an operator on $l^2_{\E}(\N) \oplus l^2_{\E}(\N)$ of the form
	\begin{equation}\label{eq-Foguel}
		R_{T} = \minimatrix{S^*}{T}{0}{S}
	\end{equation}
	where $T \in B(l^2_{\E}(\N))$.
	More generally, we refer to an operator of the form
	\begin{equation}\label{eq-nthFoguel}
		R_{T,n} = \minimatrix{(S^*)^n}{T}{0}{S^n}
	\end{equation}
	as a \emph{Foguel operator of order $n$}.
\end{Definition}

Although it is clear that a Foguel operator of order $n$ with underlying
space $\E$ can be represented as a Foguel operator on the $n$-fold
direct sum $\E^{(n)}$, we maintain the option of employing the
notation \eqref{eq-nthFoguel} since it will be convenient in what follows.

The study of such operators was implicitly initiated by Foguel \cite{Foguel}, who 
provided a counterexample to a well-known conjecture of Sz.-Nagy \cite{SZ}.  Specifically,
Foguel demonstrated the existence of a power bounded operator which is not similar to a contraction.
A simplification of Foguel's construction was almost immediately produced 
by Halmos \cite{HalmosFoguel}, who explicitly introduced operators of the form 
\eqref{eq-Foguel}.

In the years since those influential papers, Foguel operators
have proved to be an almost unlimited source of counterexamples to
important conjectures.  Most strikingly,
Foguel operators (where the underlying space $\E$ is infinite dimensional and the 
operator $T$ in \eqref{eq-Foguel} is a vectorial Hankel operator) feature prominently in Pisier's
celebrated solution to Halmos' polynomially bounded operator 
problem \cite{Pisier}.  The interested reader is advised to consult
the recent text \cite[Ch. 15]{Peller}, the expository note \cite{McJ}, 
and the influential papers \cite{AleksandrovPeller, Bourgain}.

In this note, we introduce a method for computing the norm of a Foguel operator $R_T$.
Moreover, we are also concerned with the spectral analysis of the modulus $|R_T|$.
Recall that the  \emph{modulus} of a bounded operator $A$ is defined to be
the selfadjoint operator $|A| = \sqrt{A^*A}$.  We will denote the spectrum of 
$A$ by $\sigma(A)$.

We approach this problem by considering a related system of approximate antilinear eigenvalue problems 
(see Lemma \ref{LemmaSpectrum}).
This technique, introduced in \cite[Thm. 2]{AAEPRI},  can be thought of as 
a complex symmetric adaptation of Weyl's criterion \cite[Thm. VII.12]{RS}.  To be more precise,
we require a few brief words concerning complex symmetric operators.

\begin{Definition}
	Let $\h$ denote a separable complex Hilbert space.
	A \emph{conjugation} is a conjugate-linear operator $C:\h \rightarrow \h$, 
	which is both \emph{involutive} (i.e., $C^2 = I$) and \emph{isometric} (i.e., $\inner{Cx,Cy} = \inner{y,x}$ for 
	all $x,y \in \h$).
\end{Definition}

\begin{Definition}
	We say that $T \in B(\h)$ is \emph{$C$-symmetric}
	if $T = CT^*C$ and \emph{complex symmetric} if there exists a conjugation $C$ with respect to which $T$
	is $C$-symmetric.
\end{Definition}

It turns out that $T$ is a complex symmetric operator if and only if $T$ is unitarily 
equivalent to a symmetric matrix with complex entries,
regarded as an operator acting on an $l^2$-space of the appropriate dimension (see 
\cite[Sect. 2.4]{CCO} or \cite[Prop. 2]{CSOA}).
The class of complex symmetric operators includes all normal operators, 
operators defined by Hankel matrices, compressed Toeplitz operators 
(including finite Toeplitz matrices and the compressed shift), 
and the Volterra integration operator.  For further details, we 
refer the reader to \cite{CSOA, CSO2,  CCO}.
Other recent articles concerning complex symmetric operators include \cite{Chevrot, Gilbreath, Sarason}.
    
From the perspective of this note, the key observation is that many Foguel operators (including Foguel-Hankel operators)
are complex symmetric.  Moreover, it turns out that many questions concerning the norm and
modulus of general Foguel operators can be reduced to the complex symmetric case via a block matrix
argument.  We refrain now from elaborating further on the details.  The techniques
will be explained during the course of our proofs.

\section{Main Results}

Before proceeding, let us point out a few elementary facts about the norm and modulus of
a Foguel operator.   For instance, it is clear that a Foguel operator has nontrivial kernel since
$R_{T} (e_0,0) = (0,0)$ where $e_0 = (1,0,0,\ldots)$.  In particular, this implies that
$0 \in \sigma( |R_T|)$ for any Foguel operator.
Furthermore, since $(S^*)^n$ is a compression of $R_{T,n}$, it also follows immediately that
$\norm{R_{T,n}} \geq 1$.

Our primary tool is the following lemma which characterizes the spectrum
of the modulus of a complex symmetric operator in terms of what one might call an
\emph{approximate antilinear eigenvalue problem}:

\begin{Lemma}\label{LemmaSpectrum}
	If $A$ is a bounded $C$-symmetric operator and $\lambda \neq 0$, then
	\smallskip
	\begin{enumerate}\addtolength{\itemsep}{0.5\baselineskip}\raggedright
		\item $|\lambda| \in \sigma(|A|)$ 
			if and only if there exists a sequence of unit vectors $u_n$ 
			that satisfy the approximate antilinear eigenvalue problem
			\begin{equation}\label{eq-Approximate}
				\lim_{n\rightarrow\infty} \norm{(A - \lambda C)u_n} = 0.
			\end{equation}  

		\item $|\lambda|$ is an eigenvalue of $|A|$ (i.e., a singular value of $A$) if and only if the antilinear eigenvalue problem
			$Au = \lambda Cu$ has a nonzero solution $u$. 
	  \end{enumerate}
\end{Lemma}

\begin{proof}
	The proof follows from \cite[Thm. 2]{AAEPRI}
	and the observation that \eqref{eq-Approximate} holds if and only if 
	the unit vectors $v_n = e^{ -\frac{i}{2} \arg \lambda} u_n$
	satisfy $\lim_{n\rightarrow\infty} \norm{(A - |\lambda| C)v_n} = 0$.
\end{proof}

We also require the following elementary lemma \cite[Pr.76]{Halmos}:

\begin{Lemma}\label{LemmaPair}
	If $A,B \in B(\h)$, then $\sigma(AB) \cup \{0\} = \sigma(BA) \cup \{0\}$.
	In other words, the nonzero elements of the spectra of $AB$ and $BA$ are
	the same.
\end{Lemma}

Our primary result is a type of spectral mapping theorem for the modulus of a Foguel operator.
This not only provides us with an exact formula for the norm of a Foguel operator 
(Corollary \ref{CorollaryNorm}), it also yields a sharp upper bound on the norm of powers of a 
Foguel operator (Corollary \ref{CorollaryPowers}).

\begin{Theorem}\label{TheoremMain}
	If $T \in B(l^2_{\E}(\N))$, $n \geq 0$, and 
	\begin{equation*}
		R_{T,n} = \minimatrix{(S^*)^n}{T}{0}{S^n},
	\end{equation*}
	then for $\lambda > 0$ we have
	\begin{equation}\label{eq-SpectralMapping}
		\lambda \in \sigma( |R_{T,n}| ) \backslash \{0,1\}
		  \quad \Leftrightarrow \quad |\lambda - \lambda^{-1}| \in \sigma(|T|) \backslash \{0\}.
	\end{equation}	
\end{Theorem}

\begin{proof}
	Let $S$ denote the unilateral shift on $l^2_{\E}(\N)$ and 
	let $C$ denote a conjugation on $l^2_{\E}(\N)$ which commutes with both 
	$S$ and $S^*$ (e.g., suppose that $C$ is the canonical conjugation on $l^2_{\E}(\N)$).
	We first prove the theorem under the assumption
	that $T$ is a $C$-symmetric operator (i.e., $T = CT^*C$).
	In fact, we actually prove the slightly stronger statement
	\begin{equation}\label{eq-Stronger}
		\lambda \in \sigma( |R_{T,n}| ) \backslash \{0\}
		  \quad \Leftrightarrow \quad |\lambda - \lambda^{-1}| \in \sigma(|T|)
	\end{equation}	
	for such $T$.  The reader is cautioned, however, that
	Example \ref{ExampleFails} indicates that \eqref{eq-Stronger}
	does not necessarily hold without the hypothesis that $T$ is $C$-symmetric.
	In the general case, \eqref{eq-SpectralMapping}
	will be obtained from the complex symmetric case via a block matrix argument.
	\medskip

	\noindent$(\Rightarrow)$
	Under the hypothesis that $T$ is $C$-symmetric, one can easily verify
	that the operator $R_{T,n}$ is $\widetilde{C}$-symmetric 
	with respect to the conjugation
	\begin{equation}\label{eq-BlockConjugation}
		\widetilde{C} = \minimatrix{0}{C}{C}{0}
	\end{equation}
	on $l^2_{\E}(\N) \oplus l^2_{\E}(\N)$ (for a related computation, see \cite[Ex.5]{CSO2}).
	
	If $\lambda  \in \sigma ( |R_{T,n}| ) \backslash \{0\}$, then by Lemma \ref{LemmaSpectrum}
	there exists a sequence of unit vectors 
	$u_i \in \l^2_{\E}(\N)\oplus l^2_{\E}(\N)$ that satisfy the approximate antilinear
	eigenvalue problem
	\begin{equation}\label{eq-Limit}
		\lim_{i\rightarrow\infty} \norm{ R_{T,n} u_i - \lambda \widetilde{C} u_i } =0.
	\end{equation}
	Setting $u_i = (x_i,y_i)$ and using \eqref{eq-Limit},
	we see that
	\begin{equation}\label{eq-BigLimit}
		\lim_{i\rightarrow\infty} 
		\left\|  \twovector{ (S^*)^n x_i + Ty_i - \lambda Cy_i}{S^n y_i - \lambda Cx_i} \right\| 
		= 0.
	\end{equation}
	In particular, the components of \eqref{eq-BigLimit} tend to zero separately:
	\begin{align}
		\lim_{i\rightarrow\infty}  \norm{ (S^*)^n x_i + Ty_i - \lambda Cy_i} &= 0, \label{eq-Condition1}\\
		\lim_{i\rightarrow\infty}  \norm{ S^n y_i - \lambda Cx_i} &= 0.\label{eq-Condition2}
	\end{align}
	We may assume that the sequence $\norm{y_i}$ is bounded below since otherwise 
	\eqref{eq-Condition2} would assert that some subsequence of the sequence	
	$u_i = (x_i,y_i)$ of unit vectors converges to zero, a contradiction.

	Since $C$ is isometric and commutes with $S^*$, we also deduce from \eqref{eq-Condition2} that
	\begin{equation}\label{eq-Chain}
		\lim_{i\rightarrow\infty}  \norm{ (S^*)^n x_i - \lambda^{-1}C y_i} = 0.
	\end{equation}
	At this point  \eqref{eq-Condition1}, \eqref{eq-Chain}, and 
	the Triangle Inequality yield
	\begin{equation}\label{eq-Voila}
		\lim_{i\rightarrow\infty} \norm{Ty_i - (\lambda - \lambda^{-1}) Cy_i} = 0.
	\end{equation}
	Upon normalizing the vectors $y_i$ and appealing to Lemma \ref{LemmaSpectrum}, 
	it follows that $|\lambda - \lambda^{-1}| \in \sigma(|T|)$, as desired.
	\medskip
	
	\noindent$(\Leftarrow)$
	Now suppose that $\lambda  > 0$ and $|\lambda - \lambda^{-1}| \in \sigma(|T|)$.
	It follows easily from Lemma \ref{LemmaSpectrum} that there exists a sequence $y_i$ of unit vectors so that
	\eqref{eq-Voila} holds.  Let 
	\begin{equation}\label{eq-Define}
		x_i = \lambda^{-1} CS^n y_i
	\end{equation}
	for all $i \in \N$ and observe that \eqref{eq-Condition2} holds trivially.  Furthermore,
	\eqref{eq-Voila} and \eqref{eq-Define} together imply that \eqref{eq-Condition1} also holds.
	Upon normalizing the vector $u_i = (x_i,y_i)$ and noting that \eqref{eq-Limit}
	is satisfied, we conclude that $\lambda \in \sigma( |R_{T,n} |) \backslash \{0\}$.
	This concludes the proof of the theorem in the case that $T$ is $C$-symmetric.	
	\medskip
	
	We will use the preceding special case to prove that 
	\eqref{eq-SpectralMapping} holds in general.
	Suppose that $T \in B(l^2_{\E}(\N))$ and note that
	$(CAC)^* = CA^*C$ holds for any $A \in B(l^2_{\E}(\N))$.
	We then observe that the operator
	\begin{equation*}
		\widetilde{T} = \minimatrix{T}{0}{0}{CT^*C}
	\end{equation*}
	is $\widetilde{C}$-symmetric with respect to the conjugation \eqref{eq-BlockConjugation}
	on $l^2_{\E}(\N) \oplus l^2_{\E}(\N)$.
	Since 
	\begin{equation*}
		\sigma(T^*T) \cup \{0\}  = \sigma(TT^*) \cup \{0\}
	\end{equation*}
	by Lemma \ref{LemmaPair}, it follows immediately from the definition of $\widetilde{T}$ that
	\begin{equation}\label{eq-SameSpectrum}
		\sigma( |T| ) \cup \{0\} = \sigma( |\widetilde{T}| ) \cup \{0\}.
	\end{equation}
	Letting $\widetilde{S} = S \oplus S$\
	denote the unilateral shift on $l^2_{\E}(\N) \oplus l^2_{\E}(\N)$, a short computation
	reveals that the Foguel operator	
	\begin{equation}\label{eq-SmallBlock}
		R_{\widetilde{T},n}
		=\minimatrix{ (\widetilde{S}^*)^n }{ \widetilde{T} }{0}{\widetilde{S}^n}
	\end{equation}
	is $(\widetilde{C} \oplus \widetilde{C})$-symmetric.  
	In light of the obvious isomorphism between $l^2_{\E}(\N) \oplus l^2_{\E}(\N)$ and
	$l^2_{\E\oplus \E}(\N)$, it follows from
	the first portion of this proof that 
	\begin{equation}\label{eq-Key03}
		\lambda \in \sigma( |R_{\widetilde{T},n}|) \backslash \{0\}
		\quad\Leftrightarrow\quad |\lambda - \lambda^{-1}| \in \sigma( |\widetilde{T}|).
	\end{equation}
	In particular, we note that $1\in \sigma( |R_{\widetilde{T},n}|) \backslash \{0\}$
	if and only if $0 \in \sigma( |\widetilde{T}|)$.
	
	Next, we observe that \eqref{eq-SmallBlock} may be written in the form
	\begin{equation*}
		R_{\widetilde{T},n}=
		\left(
		\begin{array}{cc|cc}
			(S^*)^n & 0 & T & 0 \\
			0 & (S^*)^n & 0 & CT^*C \\
			\hline
			0 & 0 & S^n & 0 \\
			0 & 0 & 0 &S^n
		\end{array}
		\right),
	\end{equation*}
	where each entry is an operator on $l^2_{\E}(\N)$.
	The preceding is clearly unitarily equivalent to the operator
	\begin{equation*}
		\left(
		\begin{array}{cc|cc}
			(S^*)^n & T & 0 & 0 \\
			0 & S^n & 0 & 0 \\
			\hline
			0 & 0 & (S^*)^n & CT^*C\\
			0 & 0 & 0 &S^n
		\end{array}
		\right) 
		=
		R_{T,n} \oplus R_{CT^*C,n}
	\end{equation*}
	whence we see that
	\begin{equation}\label{eq-Key01}
		\sigma( | R_{\widetilde{T},n} |)  = \sigma( |R_{T,n}| ) \cup \sigma( |R_{CT^*C,n} | ).
	\end{equation}
	On the other hand, a straightforward computation reveals that
	\begin{equation*}
		R_{T,n} = \widetilde{C} R_{CT^*C,n}^* \widetilde{C}
	\end{equation*}
	from which it follows that
	\begin{equation*}
		R_{T,n}^* R_{T,n} = \widetilde{C} (R_{CT^*C,n} R_{CT^*C,n}^*) \widetilde{C}.
	\end{equation*}
	Since $\widetilde{C}$ is a conjugation on $l^2_{\E}(\N) \oplus l^2_{\E}(\N)$ it follows easily 
	from the preceding equation that
	$\sigma(R_{T,n}^* R_{T,n}) = \sigma(R_{CT^*C,n} R_{CT^*C,n}^*)$.
	However, Lemma \ref{LemmaPair} asserts that the nonzero elements of the spectra of
	$R_{CT^*C,n} R_{CT^*C,n}^*$ and $R_{CT^*C,n}^* R_{CT^*C,n}$ are the same whence
	\begin{equation}\label{eq-Key02}
		\sigma( |R_{T,n}| ) = \sigma( |R_{CT^*C,n}| )
	\end{equation}
	since $0$ belongs to the spectrum of the modulus of any Foguel operator.
	
	Putting \eqref{eq-Key01} and \eqref{eq-Key02}  together we find that	
	$\sigma( | R_{\widetilde{T},n} |) = \sigma( |R_{T,n}| )$.
	Thus by \eqref{eq-SameSpectrum} and \eqref{eq-Key03} we obtain
	\eqref{eq-SpectralMapping} as desired.
\end{proof}

\begin{Corollary}\label{CorollarySpectra}
	 $\sigma( |R_{T,n}| )$ is discrete if and only if $\sigma(|T|)$ is discrete.
\end{Corollary}

The behavior of the spectral mapping function $f(x) = |x - x^{-1}|$ is
illustrated in Figure \ref{FigureGraph}.  In particular, notice the sharp
transition at $x = 1$.
\begin{figure}[htb!]
	\begin{center}
		\includegraphics[width=2.5in]{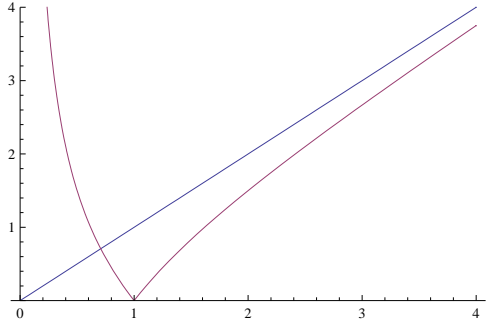}
		\caption{Graphs of $y = x$ and $y = |x - x^{-1}|$ }
		\label{FigureGraph}
	\end{center}
\end{figure}
Moreover, we see that elements in the spectrum of $|R_{T,n}|$
which are small and positive correspond to very large elements of the
spectrum of $|T|$.  

\medskip

As an immediate corollary of Theorem \ref{TheoremMain}, we have
an exact formula for the norm of a Foguel operator.
Somewhat surprisingly, the norm is independent of the parameter $n$.

\begin{Corollary}\label{CorollaryNorm}
	If $T \in B(l_{\E}^2(\N))$ and $n \geq 0$, then
	\begin{equation}\label{eq-Exact}
		\norm{R_{T,n}} = \frac{ \norm{T} + \sqrt{ \norm{T}^2 + 4} }{2}.
	\end{equation}	
\end{Corollary}

\begin{proof}
	Since $\norm{R_{T,n}}\geq 1$, it follows from Theorem \ref{TheoremMain}
	and the monotonicity of the function $f(x) = x - x^{-1}$ on the interval $[1,\infty)$ that
	\begin{equation*}
		\norm{T} = \norm{R_{T,n}} - \norm{R_{T,n}}^{-1}.
	\end{equation*}
	Solving the preceding equation for $\norm{R_{T,n}}$ yields
	the desired formula \eqref{eq-Exact}
	(note that the positive root of the resulting quadratic 
	is taken since $\norm{R_{T,n}}  \geq 1$).
\end{proof}

\begin{Example}
	Let us consider the operator introduced by Halmos \cite{HalmosFoguel}
	in his simplification of Foguel's solution \cite{Foguel} to Sz.-Nagy's problem \cite{SZ}.
	Let $\J = \{ 3^n : n  \in \N\}$ and observe that $\J$ has the property that
	$i,j \in \J$ and $i < j$ together imply that $2i < j$.
	Let $P: l^2(\N) \rightarrow l^2(\N)$ be the orthogonal projection
	onto the closed linear span of $\{ e_j: j \in \J\}$ and let	
	\begin{equation*}
		R_P = \minimatrix{S^*}{P}{0}{S}.
	\end{equation*}
	A short combinatorial argument shows that
	\begin{equation*}
		R^n_P = \minimatrix{(S^*)^n}{P_n}{0}{S^n}
	\end{equation*}
	for $n \geq 1$,	where 
	\begin{equation*}
		P_n = \sum_{j=0}^{n-1} (S^*)^j P S^{n-1-j}
	\end{equation*}
	is a \emph{partial isometry} (this is due to the sparseness of the set $\J$).  
	Since $\norm{ P_n} = 1$, it follows from
	Corollary \ref{CorollaryNorm} that
	\begin{equation*}
		\norm{  R^n_P} = \frac{1 + \sqrt{5}}{2}
	\end{equation*}
	for $n \geq 1$.  In particular, Halmos' operator $R_P$ is more than just
	power bounded -- its positive integral powers all have norm equal to the Golden Ratio.
\end{Example}

Along similar lines, 
Theorem \ref{TheoremMain} also permits sharp upper bounds on the norms of \emph{powers}
of Foguel operators:

\begin{Corollary}\label{CorollaryPowers}
	If $T \in B(l^2_{\E}(\N))$ and $n \geq 1$, then
	\begin{equation}\label{eq-InequalityGeneral}
		\norm{ R_{T}^n} \leq \frac{ n\norm{T} + \sqrt{ n^2\norm{T}^2 + 4 }}{2}.
	\end{equation}
	The inequality \eqref{eq-InequalityGeneral} is sharp in the sense that
	for each $n \geq 1$, there exists an operator $T$ for which equality is attained.	
\end{Corollary}

\begin{proof}
	By induction, we have
	\begin{equation*}
		R_T^n = \minimatrix{S^*}{T}{0}{S}^n = \minimatrix{(S^*)^n}{  T_n }{0}{S^n}
	\end{equation*}
	where
	\begin{equation}\label{eq-Sum}
		T_n = \sum_{j=0}^{n-1} (S^*)^j T S^{n-1-j}.
	\end{equation}
	Since \eqref{eq-Sum} implies that $\norm{T_n} \leq n \norm{T}$, it follows from \eqref{eq-Exact} that
	\begin{equation}\label{eq-Attained}
		\norm{R_T^n}
		= \frac{ \norm{T_n} + \sqrt{ \norm{T_n}^2 + 4 }}{2} 
		\leq  \frac{ n\norm{T} + \sqrt{ n^2\norm{T}^2 + 4 }}{2}, 
	\end{equation}
	which is the desired inequality \eqref{eq-InequalityGeneral}.
			
	We claim that equality is attained in \eqref{eq-InequalityGeneral} if $T = I$.
	In this case, \eqref{eq-Sum} reduces to
	\begin{align*}
		T_n 
		&= \sum_{j=0}^{n-1} (S^*)^j S^{n-1-j}\\
		&= S^{n-1} + S^{n-3} + \cdots + (S^*)^{n-3} + (S^*)^{n-1}
	\end{align*}
	whence $\norm{T_n} \leq n$.  On the other hand, letting
	\begin{equation*}
		k_r =  \sqrt{1 - r^2}\,(1,\, r ,\, r^2,\,\ldots)
	\end{equation*}
	for $r \in [0,1)$ and noting that each operator
	$T_n$ is selfadjoint, we find that
	\begin{align*}
		\lim_{r\rightarrow 1^-} \inner{T_n k_r,k_r}
		&= \lim_{r\rightarrow 1^-} \left( \inner{S^{n-1}k_r,k_r} + \inner{S^{n-3}k_r,k_r} 
			+ \cdots + \inner{(S^*)^{n-1}k_r,k_r} \right)\\
		&= \lim_{r\rightarrow 1^-} (r^{n-1} + r^{n-3} + \cdots + r^{n-3} + r^{n-1} )\\
		&= n
	\end{align*}
	which shows that $\norm{T_n} =n$.  It follows from \eqref{eq-Attained} and the fact that 
	$\norm{T} = \norm{I} = 1$ that equality holds in \eqref{eq-InequalityGeneral}.
\end{proof}

\begin{Example}
	If $T \in B( l^2(\N))$ is a contraction, we have the sharp inequality
	\begin{equation*}
		\norm{ \minimatrix{S^*}{T}{0}{S}^n} \,\leq\, \frac{n + \sqrt{n^2+4}}{2}.
	\end{equation*}
	In particular,
	\begin{equation*}
		\norm{ \minimatrix{S^*}{T}{0}{S}} \,\leq\, \frac{1 + \sqrt{5}}{2}.
	\end{equation*}
\end{Example}

We conclude with two examples which illustrate the relationship between conditions
\eqref{eq-SpectralMapping} and \eqref{eq-Stronger}. 
The following example shows that \eqref{eq-SpectralMapping} cannot in general be
improved to \eqref{eq-Stronger} if the operator $T$ is not $C$-symmetric:

\begin{Example}\label{ExampleFails}
	It is well known and easy to show that the unilateral shift $S$
	is not a complex symmetric operator
	 \cite[Ex. 2.14]{CCO}, or \cite[Cor. 7]{MUCFO}.
	Let us consider the Foguel operator $R_S$.
	
	Since $S^*S = I$, it is clear that 
	$0 \notin \sigma(|S|)$.  On the other hand, a short computation reveals that
	$1 \in \sigma(|R_S|)$.  Indeed, observe that
	\begin{equation*}
		R_S^* R_S = \minimatrix{S}{0}{S^*}{S^*} \minimatrix{S^*}{S}{0}{S}
		= \minimatrix{SS^*}{S^2}{(S^*)^2}{2I}
	\end{equation*}
	and 
	\begin{equation*}
		\minimatrix{SS^*}{S^2}{(S^*)^2}{2I} \twovector{e_1}{0} = \twovector{e_1}{0}
	\end{equation*}
	where $e_1 = (0,1,0,0,\ldots)$.  In particular, the fact that $0 \notin \sigma(|S|)$
	while $1 \in \sigma(|R_S|)$ indicates that \eqref{eq-Stronger} does not
	hold in general.   
	The astute reader will observe that the trouble stems 
	from \eqref{eq-SameSpectrum}, which cannot in general be improved to yield 
	$\sigma(|T|) = \sigma( |\widetilde{T}|)$ for arbitrary $T \in B(\h)$.
\end{Example}

\begin{Example}
	Unlike the unilateral shift, the identity operator is $C$-symmetric with respect to the canonical
	conjugation on $l^2(\N)$.
	According to \eqref{eq-Stronger}, it follows that 
	$1 \notin \sigma(|R_I |)$ since
	$0 = 1-1^{-1}\notin \sigma(I)$.
	This can be verified directly by using the formula
	\begin{equation}\label{eq-RR}
		R_I^* R_I = \minimatrix{S}{0}{I}{S^*} \minimatrix{S^*}{I}{0}{S}
		= \minimatrix{SS^*}{S}{S^*}{2I}
	\end{equation}
	and Weyl's criterion \cite[Thm. VII.12]{RS}.  
	Indeed, suppose toward a contradiction that there exists a sequence of unit vectors
	$u_n = (x_n,y_n)$ such that 
	\begin{equation*}
		\lim_{n\rightarrow\infty} ( |R_I|u_n - u_n ) = 0.
	\end{equation*}
	By \eqref{eq-RR}, this is equivalent to asserting that
	\begin{align}
		\lim_{n\rightarrow\infty} (SS^*x_n + Sy_n - x_n) &=0, \label{eq-RR01} \\
		\lim_{n\rightarrow\infty} (S^*x_n + y_n)   &=0. \label{eq-RR02}
	\end{align}
	Applying $S$ to \eqref{eq-RR02}, we find that
	$\lim_{n\rightarrow\infty} (SS^*x_n + Sy_n) = 0$ from which 
	$x_n \to 0$ follows in light of \eqref{eq-RR01}.  Since $x_n \to 0$, we conclude from 
	\eqref{eq-RR02} that $y_n \to 0$ as well.  However, this implies that the sequence of 
	unit vectors $u_n = (x_n,y_n)$ tends to zero, a contradiction.
\end{Example}

\end{document}